\documentclass[11pt]{article}

\usepackage{subfigure}
\usepackage{graphicx}
\usepackage{amsmath, amsthm, amssymb}
\usepackage{array}

\usepackage[margin=3cm]{geometry}

\newtheorem{theorem}{Theorem}[section]
\newtheorem{lemma}[theorem]{Lemma}
\newtheorem{proposition}[theorem]{Proposition}

\newtheorem{remark}[theorem]{Remark}

 \newcommand{\scalprod}[2]{\left\langle #1,#2 \right\rangle}

\begin{document}

\title{Quiet sigma delta quantization, and global convergence for a class of asymmetric piecewise affine maps}

\author{Rachel Ward  \thanks{R. Ward is an NSF postodoctoral fellow at the Courant Institute, New York University, New York, New York.  e-mail: rward@cims.nyu.edu}}

\maketitle

\begin{abstract}
\noindent In this paper, we introduce a family of second-order sigma delta quantization schemes for analog-to-digital conversion which are `quiet': quantization output is guaranteed to fall to zero at the onset of vanishing input.  In the process, we prove that the origin is a globally attractive fixed point for the related family of asymmetrically-damped piecewise affine maps.  Our proof of convergence is twofold:  first, we construct a trapping set using a Lyapunov-type argument; we then take advantage of the asymmetric structure of the maps under consideration to prove convergence to the origin from within this trapping set.

\end{abstract}

\section{Introduction}
Analog-to-digital conversion is the study of accurate and tractable methods for the approximation of real-valued signals using a finite alphabet.  It is of great importance as many signals of interest such as audio, are naturally produced in analog form, while it is becoming increasingly efficient to store and manipulate information in digital format.

The process of analog-to-digital conversion usually consists of two parts: \emph{sampling} and \emph{quantization}.  Sampling consists of converting the continuous-time signal of interest $f(t)$ into a discrete-time signal $f(t_n)$, and quantization is the process of mapping the discrete-time signal $f(t_n)$ into a sequence of discrete values, in such a way that the original function $f(t)$ can be reconstructed, albeit imperfectly, from these discrete values at a later time.

\paragraph{Sampling.}

In the setting of analog-to-digital conversion, the signal of interest $f(t)$ is often modeled as a bounded, \emph{bandlimited} function.  Specifically, the signal is assumed to belong to the space ${\cal B}_{\Omega}$ of real-valued continuous functions that are bounded in $L^{\infty}(\mathbb{R})$, and whose Fourier transforms (as distributions) have support contained in $[-\Omega/2, \Omega/2]$ for a known bandwidth $\Omega$.  For example, speech signals can be modeled as bandlimited functions whose bandwidth $\Omega$ is $4$ KHz, and audio signals in general are well-modeled as bandlimited functions with bandwidth $20$ KHz.  For ease of presentation, let us fix $\Omega = 1$ in the sequel; all of our results can be extended to general bandwidths by change of variables.  For the Fourier transform, we shall use the normalization 
\begin{equation}
\widehat{f}(\omega) := \int_{-\infty}^{\infty} f(t) \exp{(-2 \pi i\omega t)} dt
\label{fourier}
\end{equation}
for $f \in L^1(\mathbb{R})$, and extended to the space of tempered distributions in the usual way.

For bandlimited functions $f$, the low frequency content of the sequence of samples $\big( f (\frac{n}{\lambda}) \big)_{n \in \mathbb{Z}}$ is the function itself; therefore, such functions can be reconstructed using a low-pass filter.  In mathematical terms, this intuition corresponds to the Shannon-Nyquist sampling theorem: if $\lambda > 1$, then any function $f \in {\cal{B}}_1$ can be recovered as a weighted sum of translates of an averaging kernel $g \in L^1(\mathbb{R})$ via the formula
\begin{equation}
f(t) = \frac{1}{\lambda} \sum_{n \in \mathbb{Z}} f \Big( \frac{n}{\lambda} \Big) g \Big( t - \frac{n}{\lambda} \Big), 
\label{reconstruct}
\end{equation}  
where $g$ is any bounded and continuous function whose Fourier transform $\widehat{g}$ satisfies
\begin{eqnarray}
\label{g}
\widehat{g}(\omega) &=&  \left\{ \begin{array}{ll}1, & \textrm{if } | \omega | \leq 1/2  \\
0, & \textrm{if } | \omega | \geq \lambda_0/2
\end{array} \right.
\end{eqnarray}
for some arbitrary $\lambda_0$ with $\lambda \geq \lambda_0 > 1$.  


\paragraph{Quantization.}  Given a sequence of samples $\big( f (\frac{n}{\lambda}) \big)_{n \in \mathbb{Z}}$ associated to a bandlimited function $f \in {\cal B}_1$, a $K$-level quantization scheme assigns a sequence of \emph{quantized values} $q_n^{\lambda}$ from an alphabet $\cal{A}_K$ of size $| {\cal A}_K | = K$ in such a way that the function
\begin{equation}
\label{eq:fapprox}
\widetilde{f}_{\lambda}(t)= \frac{1}{\lambda} \sum_{n \in \mathbb{Z}} q^{\lambda}_n g \big( t - \frac{n}{\lambda} \big)
\end{equation}
serves as a good approximation to $f(t) = \frac{1}{\lambda} \sum_{n \in \mathbb{Z}} f \big( \frac{n}{\lambda} \big) g \big( t - \frac{n}{\lambda} \big)$.
A natural choice for the discrete coefficient $q_n^{\lambda}$ in \eqref{eq:fapprox} is the truncated $K$-$1$ term binary expansion of the amplitude $| f^{\lambda}_n| $, multiplied by the sign of $f^{\lambda}_n$ - such \emph{binary} quantization schemes are an industry standard for digitizing audio signals.  However, binary quantization suffers from various implementation difficulties and disadvantages in practice, mostly related to the cost of building analog circuits that can carry out binary expansions accurately to many digits; in practice, it is not always the quantization scheme of choice among engineers.  

In many applications,  \emph{oversampled coarse quantization} is instead preferred, where a fixed number of levels --- sometimes as few as two levels, $q^{\lambda}_n \in \{-1,1\}$ --- is allocated to each sample $f^{\lambda}_n$ at the more tolerated expense of very high sampling rate compared to the rate $\lambda \approx 1$ which is sufficient for binary quantization.  From the viewpoint of circuit engineering, oversampled coarse quantization is associated to low-cost analog hardware, because increasing the sampling rate is cheaper than refining the quantization.   Further advantages of oversampled coarse quantization methods include a built-in redundancy and robustness against errors resulting from imperfections in the analog circuit implementation.  This robustness comes as a consequence of the more `democratic' distribution of bit significance in the reconstruction formula \eqref{eq:fapprox}; we refer the reader to \cite{DC02} for more details. 

\subsection{Sigma Delta quantization}

\noindent  In sigma delta ($\Sigma \Delta$)  quantization, one of the most widely-used oversampled quantization methods in practice, $q^{\lambda}_n$ is dynamically updated as a function of previous $q^{\lambda}_k$ and $f^{\lambda}_k$ in such a way that the frequency content of the quantization error is pushed to high frequencies; these high-frequency error components are then cancelled out by the convolution, \eqref{eq:fapprox}, which acts as a low-pass filter.   In mathematical terms, $q^{\lambda}_n$ being a high-pass sequence means that the difference $f^{\lambda}_n - q^{\lambda}_n$ is the $m$th order difference of a bounded auxiliary sequence $v_n$.   Precisely, for an $m$th order $\Sigma \Delta$ quantization scheme, the sequence of coefficients $q_n^{\lambda} \in {\cal A}$ is chosen such that
\begin{equation}
\label{eq:relation}
f^{\lambda}_n - q^{\lambda}_n = ( \Delta^m v)_n := \sum_{l=0}^m (-1)^l {m \choose l} v_{n+m-l}, \hspace{5mm} n \in \mathbb{Z},
\end{equation}
The quantization error associated to such quantization schemes may be bounded as follows: 

\begin{proposition}
\label{prop:dd1}
Fix $f \in {\cal B}_1$, and furthermore suppose that $\| f \|_{\infty} \leq \alpha < 1$ for some $\alpha < 1$. Fix a filter $g$ satisfying the assumptions \eqref{g}, and suppose further that $g$ and its first $m$ derivatives belong to  $L^1(\mathbb{R})$.  Fix an oversampling ratio $\lambda \geq \lambda_0$, and let ${\cal A}$ be a discrete alphabet consisting of equispaced values with endpoints $+1$ and $-1$.   Suppose that there exists a sequence of coefficients $q^{\lambda}_n \in {\cal A}$ and a bounded state sequence $(v_n)_{n =0}^{\infty}$ for which the $m$th order relation \eqref{eq:relation} is satisfied.  Then one may set $t_0 = t_0(m,\lambda) > 0$ sufficiently large that for all $t \geq t_0$,
\begin{equation}
\label{order}
\Big|  f(t) - \frac{1}{\lambda} \sum_{n = 1}^{\infty} q^{\lambda}_n  g \Big( t - \frac{n}{\lambda} \Big) \Big| \leq  C \lambda^{-m};
\end{equation}
the constant $C$ appearing above depends only on $\| v \|_{\infty}$ and on $\| g \|_{L^1(\mathbb{R})}, \| g^{(1)} \|_{L^1(\mathbb{R})},  ...  \| g^{(m)} \|_{L^1(\mathbb{R})}$.
\end{proposition}

For a rigorous proof of Proposition \ref{prop:dd1}, we refer the reader to \cite{DD03}.   Note that in the error estimate \eqref{order}, the index $n$ ranges over the positive integers only; this change of setup yields no difficulties, given that a filter $g$ satisfying \eqref{g} can be made to be well-localized in time by requiring that its Fourier transform be sufficiently smooth, see \cite{S} for more details.

\begin{remark}
We shall refer to any quantization scheme satisfying an error estimate of the form \eqref{order} as an \emph{$m$th order quantization scheme}.  
\end{remark}


\paragraph{Construction of $\Sigma \Delta$ quantization schemes.}

We now turn to the issue of existence of quantization schemes satisfying the difference relation \eqref{eq:relation}.  Although high-order $\Sigma \Delta$ quantizers have been implemented in practice for many years \cite{SchTemes}, the construction of $\Sigma \Delta$ schemes of arbitrary order $m \geq 1$ for which the state sequence $(v_n)_{n \in \mathbb{N}}$ is guaranteed to remain bounded has been only recently achieved by Daubechies and DeVore in \cite{DD03}:

\begin{proposition}[Daubechies and DeVore]
\label{prop:dd}
Suppose that $f$, $g$, and ${\cal A}$ satisfy the assumptions of Proposition \ref{prop:dd1}.  Take $\lambda > 1$, and consider a quantizer function $Q: [-\alpha, \alpha] \rightarrow {\cal A}$ that satisfies
\begin{eqnarray}
\label{eq:quant}
Q(u) &=& \textrm{sign}(u) \hspace{4mm} \textrm{for } |u| \geq 1/2. \nonumber 
\end{eqnarray}
Then there exist admissible functions $F: \mathbb{R}^m \rightarrow \mathbb{R}$ for which the recursion 
\begin{eqnarray}
\label{eqn:sd}
q_n &=& Q \Big( F(f_n, u_{n-1}^{(1)}, ..., u_{n-1}^{(m)}) \Big) \nonumber \\
u_n^{(1)} &=& u_{n-1}^{(1)} + f_n - q_n \nonumber \\
u_n^{(j)} &=& u_{n-1}^{(j)} + u_n^{(j-1)}, \hspace{5mm} j=2, ..., m 
\end{eqnarray}
with initial conditions $u_0^{(j)} = 0$ for $j=1, ..., m$, generates a sequence $(v_n)_{n \in \mathbb{N}} := ( u_n^{(m)} )_{n \in \mathbb{N}}$ that is bounded and satisfies the finite difference equation \eqref{eq:relation}.

\end{proposition}

A few remarks are in order.

\begin{enumerate}

\item The recursions \eqref{eqn:sd} are not the only means for generating coarse quantization schemes having  $m$th order accuracy \eqref{order};  therefore, we shall refer to them as the \emph{standard} $\Sigma \Delta$ recursions.  For a comprehensive overview on more general setup for coarse quantization, we refer the reader to \cite{felix_thesis}.

\item In the case $m=1$, one may take $F(f_n, v_{n-1}) = f_n + v_{n-1}$ as admissible function in \eqref{eqn:sd}, and the recursion reduces to
\begin{equation}
\label{first-order}
v_{n} = v_{n-1} + f_n -Q(v_{n-1} + f_n).
\end{equation}
In this case, one may verify the boundedness of $(v_n)_{n \in \mathbb{N}}$ by induction:  if $| f_n | \leq 1$ for all $n \geq 1$, and if $| v_0 | \leq 3/2$, then $|v_n | \leq 3/2$ for all $n \geq 0$. 

 \item For $m=2$, the admissible functions of the form $F_{\gamma}(f_n, u_{n-1}, v_{n-1}) = \gamma u_{n-1} +  v_{n-1}$ guarantee boundedness of the sequence $(v_n)_{n \in \mathbb{N}}$ and $\| v \|_{\infty} \leq C_{\alpha}$, with the constant $C_{\alpha}$ depending only on $| f_n | \leq \alpha < 1$, for a range $\gamma= \gamma(\alpha) \geq 1$; see \cite{OYthesis} for more details. 
 
 \item For higher orders $m \geq 3$, the only admissible functions $F = F_m$ that have been \emph{proven} to guarantee boundedness of the $(v_n)_{n \in \mathbb{N}}$ for recursions of the form \eqref{eqn:sd} were constructed in \cite{DD03}; they are defined recursively with respect to the order $m$, and become increasingly complicated with increasing order.   Nevertheless, stability has been shown for a different class of recursions generating $m$th order quantization schemes, and the criteria for stability there are more aligned with the schemes implemented in practice \cite{felix_thesis}.  Still, proving stability for a wider class of admissible functions for orders $m \geq 3$ remains a challenging open problem of interest to both mathematicians and engineers.

\end{enumerate}

\subsection{Quiet $\Sigma \Delta$ quantization: defeating infinite memory}
As previously mentioned, the input signals in the context of analog-to-digital conversion are well-modelled as bandlimited functions.   Nevertheless, audio signals in actuality have finite time support, completely vanishing over intervals of time such as between speech phrases or musical tracks, and ultimately, vanishing indefinitely.  As bandlimited functions are restrictions of analytic functions to the real line, such functions cannot vanish identically on an interval unless they are identically zero; nevertheless, bandlimited functions can become arbitrarily small over arbitrarily large intervals of time, and still provide a good model for audio signals.

For efficiency reasons, it is desirable that the quantization sequence $(q_n)_{n \in \mathbb{N}}$, and in turn the state sequence $(v_n)_{n \in \mathbb{N}}$, fall to zero in response to vanishing input $f_n = f(\frac{n}{\lambda}) = 0$.
For instance, as the quantization level $q_n = 0$ corresponds to zero-voltage level, quantization schemes as such can essentially `shut off' over the off-support of the input, rendering them \emph{low-power}.   Indeed, there are devices that drive speakers with very high efficiency which use a three-level quantization alphabet  $q_n \in \{-1, 0, 1\}$, dating back to a patent issued to Crystal Semiconductor, Inc. (now Cirrus logic), see \cite{patent1} for more details.  For efficiency reasons as explained in that patent, it is desirable to maximize the number of `$0$'s and minimize the number of $+1$'s and $-1$'s in the tri-level quantization output.  In particular, devices which implement $\Sigma \Delta$ quantization with tri-level quantization are in widespread use in modern electronic technology \cite{adams}, see the website of Analog Devices for more details.



 \begin{remark}
We shall refer to any  quantization scheme for which the quantization sequence $(q_n)_{n \in \mathbb{N}}$, and the state sequence $(v_n)_{n \in \mathbb{N}}$ fall to zero at the onset of vanishing input $f_n = 0$ as \emph{quiet}.   
 \end{remark}
 
 Unfortunately, the quantization output produced by the standard $\Sigma \Delta$ schemes, \eqref{eqn:sd} is \emph{not} quiet; in contrast, quantization sequences $q_n$ produced by such recursions generally fall into periodic cycles at the onset of vanishing input. 
The treatment of such zero-input periodicities, which can cause spurious idle tones in the reconstructed signal, has been an area of active research over the past twenty years, see the patent \cite{patent2}, and also the papers  \cite{ChaosDither}, \cite{SDChaos} for more details.  One of the methods used in practice to break such cycles is to apply \emph{dither}, or random noise, to the sample input; note that $m$th-order accuracy in the sense of \eqref{order} is no longer obtained subject to the application of dither.  
%





\paragraph{Modifying the standard $\Sigma \Delta$ recursions to be quiet.}
Let us consider the task of modifying the standard $\Sigma \Delta$ recursions \eqref{eqn:sd} to be quiet, or so that ${\bf x}_n := (u_n^{(1)}, u_n^{(2)}, ... , u_n^{(m)})$ and $q_n$ fall to zero in response to vanishing input, and doing so while still maintaining $m$th order accuracy of the resulting approximations. We shall restrict attention to quantization schemes using \emph{tri-level} quantizers of the form $Q(u) = Q_{tri}(u)$, where 
\begin{equation}
Q_{tri}(u) = \left\{\begin{array}{cl}
1 & u >  1/2, \\
0 & -1/2 \leq u \leq 1/2, \\
-1 & u < -1/2. \end{array}\right.
\label{tri-level}
\end{equation}
As discussed previously, tri-level quantizers are often implemented in quantization devices which are built primarily for high efficiency, which is the motivational setting for quiet quantization.   

\paragraph{Dynamical system interpretation.} When $f_n = 0$,  the standard $m$th order recursion \eqref{eqn:sd} reduces to an $m$-dimensional piecewise-affine map. Implemented with tri-level quantizer \eqref{tri-level}, and assuming that $-1/2 < F(0, {\bf 0}) < 1/2$, the state ${\bf x} = {\bf 0}$ is a fixed point of this so-called \emph{zero-input} map.  Therefore, quietness is achieved if the standard recursions \eqref{eqn:sd} may be modified so that ${\bf x}= {\bf 0}$ is an \emph{attractive} fixed point for the zero-input map, at least within a neighborhood containing the bounded sequence $({\bf x}_n)_{n \in \mathbb{N}}$, and if this can be done while still maintaining boundedness of the sequence $(v_n)_{n \in \mathbb{N}}$, and $m$th order reconstruction accuracy \eqref{order} of the scheme.  

\paragraph{Quietness for first-order $\Sigma \Delta$ quantization.}

Quietness for the first-order $\Sigma \Delta$ scheme $m=1$, implemented with tri-level quantizer, has been previously studied previously in \cite{OYthesis}.  There, the author shows that quietness is achieved by composing the first-order recursion \eqref{first-order} with a contraction $v_n \rightarrow \rho v_n$, leading to the map $v_{n} =  \rho v_{n-1} + f_n - Q(\rho v_{n-1} + f_n)$.   Indeed, it is not hard to verify that

\begin{enumerate}
\item  the uniform bound $| v_n | \leq 3/2$ holds subject to this modification, 
\item the fixed point $v = 0$ is a globally attracting fixed point of the zero-input map $v_{n} = \rho v_{n-1} - Q_{tri}(\rho v_{n-1})$, and 
\item the modified recursion $v_{n} =  \rho v_{n-1} + f_n - Q_{tri}(\rho v_{n-1} + f_n)$ still represents a first-order scheme if the damping is such that $\rho \geq \rho_{\lambda} := 1 - \frac{1}{\lambda}$, where $\lambda$ is the oversampling ratio.  We refer the reader to \cite{OYthesis} for a proof of these results.  

%


\end{enumerate}

\paragraph{Quietness for second-order $\Sigma \Delta$ quantization.}

\noindent  Unfortunately, first-order $\Sigma \Delta$ schemes are rarely used in practice, and we would like to obtain similar quiet modifications of the higher-order standard recursions \eqref{eqn:sd}. Second-order $\Sigma \Delta$ schemes are often preferred in practice for having improved reconstruction error over first-order schemes, while still maintaining stability and ease of implementation.  For the remainder of this paper, we shall concentrate only the case $m=2$; we leave the analysis of higher-order quiet $\Sigma \Delta$ schemes to future work.

As mentioned previously, the simplest admissible functions for the second-order recursion \eqref{eqn:sd} are of the form $F(u, v) = \gamma u + v$.  For this case and for a range of $\gamma \geq 1$ that depends on the parameter $\alpha$ in the bound $|f_n| \leq \alpha < 1$, the sequence $(v_n)_{n \in \mathbb{N}}$ generated by the second-order scheme \eqref{eqn:sd} is guaranteed to remain bounded.  This is shown in \cite{OYthesis} by constructing a family of convex sets $S_{\alpha} \subset \mathbb{R}^2$ which are invariant under the recursion.  Implemented with tri-level quantizer and linear rule $F(u,v) = \gamma u + v$, the second-order recursion \eqref{eqn:sd} has the form,
\begin{eqnarray}
\label{second-order}
\nonumber
(u_0, v_0) &=& (0,0), \nonumber \\
\nonumber \\
q_n &=&  \left\{
\begin{array}{ll}
-1, & F(u_{n-1}, v_{n-1}) \leq -1/2 \nonumber \\

0, &  \Big| F(u_{n-1}, v_{n-1}) \Big| < 1/2 \nonumber \\

1, &  F(u_{n-1}, v_{n-1}) \geq 1/2, \nonumber 
\end{array}\right. \nonumber \\ \nonumber \nonumber \\
(u_{n}, v_{n})  &=& T\big( (u_{n-1}, v_{n-1}), f_n \big) := (u_{n-1} + f_n - q_n, \hspace{1mm} u_{n-1} + v_{n-1} + f_n - q_n).
\end{eqnarray}
It is natural to wonder whether the damped modification $(u_n, v_n) = T(\rho(u_{n-1}, v_{n-1}), f_n)$ might induce quietness in this case as it did for the first-order recursion \eqref{first-order}.  Indeed, the behavior of such a `leaky' modification has been studied extensively in its own right, for the behavior of its resulting periodic cycles \cite{feely}, and for its approximation accuracy \cite{OYthesis}. With respect to the latter point, it has been shown that  the boundedness of the state sequence $(v_n)_{n \in \mathbb{N}}$ is not affected by this modification, as the invariant sets for the map $(u_{n+1}, v_{n+1}) = T((u_n, v_n), f_n)$ are convex and contain the origin.  Moreover, identifying $u_n - \rho u_{n-1} = v_n - 2\rho v_{n-1} + \rho^2 v_{n-2}$, one verifies that  second-order accuracy is maintained for the leaky scheme $(u_n, v_n) = T(\rho(u_{n-1}, v_{n-1}), f_n)$ if the damping factor $\rho$ is bounded below by $\rho_{\lambda} = 1 - \frac{1}{\lambda}$ at each step:
\begin{eqnarray}
\label{2ndorder:leak}
f(t) - \widetilde{f}(t) &=& \frac{1}{\lambda} \sum_{n \geq 1} (f_n - q_n) g \Big( t - \frac{n}{\lambda}\Big) \nonumber \\
&=&  \frac{1}{\lambda} \sum_{n \geq 1} (v_n - 2\rho v_{n-1} + \rho^2 v_{n-2}) g\Big(1 - \frac{n}{\lambda}\Big) \nonumber \\
&=& \frac{1}{\lambda} \sum_{n \geq 1}(v_n - 2v_{n-1} + v_{n-2}) g\Big(t - \frac{n}{\lambda}\Big) + 2(1 - \rho) \frac{1}{\lambda}\sum_{n \geq 1} (v_{n-1}-v_{n-2}) g \Big( t - \frac{n}{\lambda}  \Big)\nonumber \\
&& +(1 - \rho)^2 \frac{1}{\lambda}\sum_{n \geq 1} v_{n-2} \cdot g \Big( t - \frac{n}{\lambda} \Big).
\end{eqnarray}
By Proposition \ref{prop:dd1}, the first term in the ultimate expression is bounded in magnitude by $ C\lambda^{-2}$ for $t \geq t_0(\lambda)$ sufficiently large, and the second term is bounded in magnitude by $C\lambda^{-1}(1 - \rho) \leq C\lambda^{-2}$.  This is because the first term corresponds to the approximation error for a standard second-order $\Sigma \Delta$ recursion, and the second term corresponds to the approximation error incurred for a standard first-order $\Sigma \Delta$ recursion, multiplied by $2(1-\rho)$.   At the same time, the third term is proportional to $(1 - \rho)^2 C= \lambda^{-2} C$ for a constant $C$ depending on $\| g \|_{L^1(\mathbb{R})}$ and $\| v \|_{\infty}$ only.   In total, we arrive at the desired second-oder bound: $| f(t) - \widetilde{f}(t) | \leq C \lambda^{-2}$ for $t \geq t_0(\lambda)$.

\paragraph{The second-order leaky modification is not quiet.} Unfortunately, despite retaining second-order accuracy, leaky modifications to the standard second-order recursion of the form $(u_n, v_n) = T(\rho(u_{n-1}, v_{n-1}), f_n)$ still do not achieve quietness, because the origin is not an attractive fixed point for the piecewise affine zero-input map $(u_n, v_n) = T(\rho_n(u_{n-1}, v_{n-1}), 0)$ in a sufficiently large neighborhood. For example, if $u_0 = ( 1 + \rho)^{-1}$, $v_0 = 0$, and $f_1 = (1 + \rho)^{-1}$ at the onset of zero input $f_2 = f_n = 0, n \geq 2$, the system falls into period-$2$ oscillation.  Nor are such period-$2$ trajectories `pathological situations' in an otherwise well-behaved neighborhood of convergent trajectories; numerical results such as those in Figure \ref{fig:quietmap} suggest that the neighborhood of attractivity for the piecewise-affine system $(u_{n+1}, v_{n+1}) = T(\rho(u_n, v_n), 0)$  shrinks to a one-dimensional subset of $\mathbb{R}^2$ as $\rho \rightarrow 1$.  

\begin{figure}[!h]
\begin{center}
\includegraphics[width=3in]{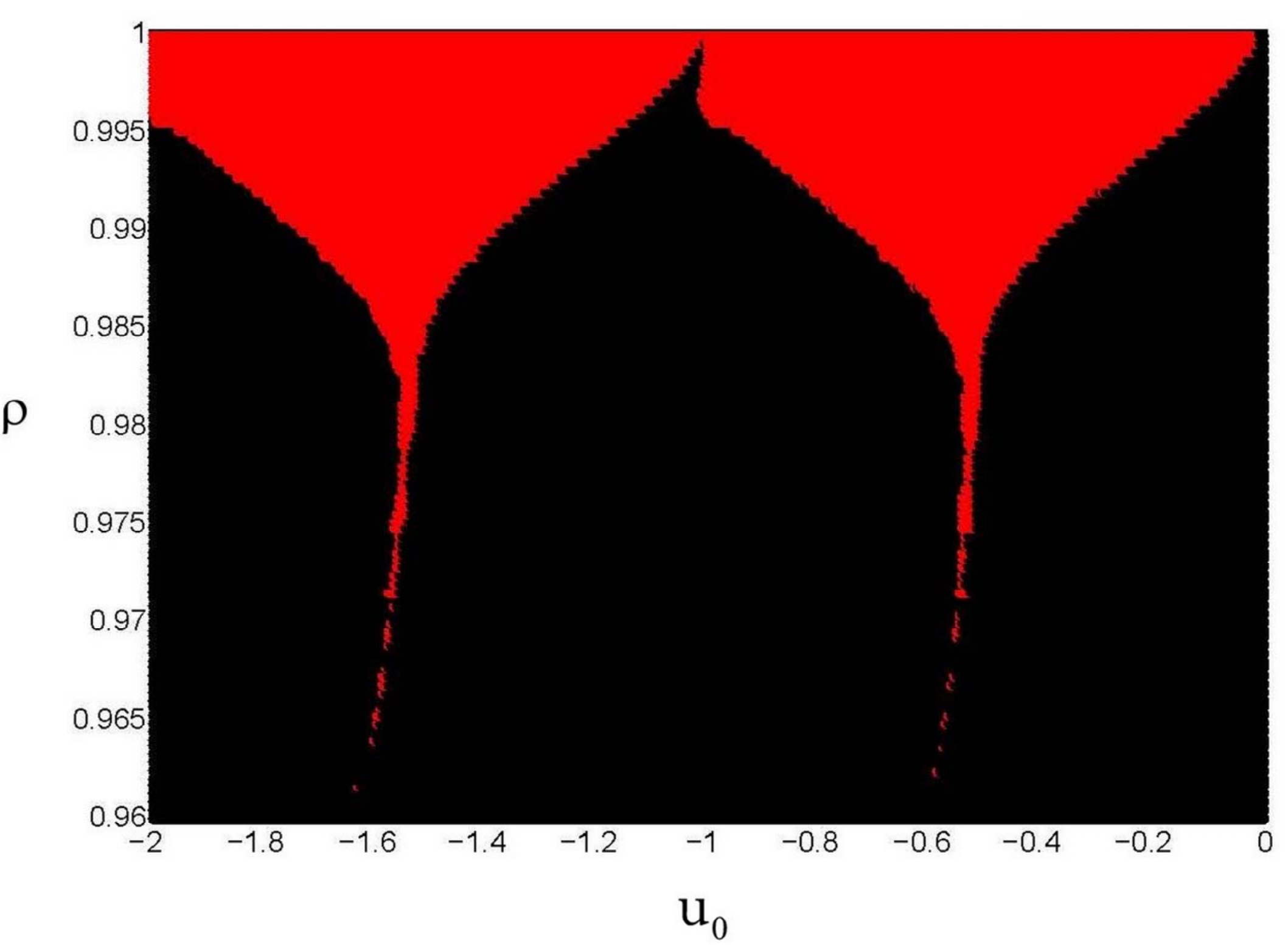}
\caption{ \small We indicate whether or not the trajectory of ${\bf x}_0 = (u_0, v_0)$ under the map $(u_{n+1}, v_{n+1})= T\big( \rho (u_n, v_n),0 \big)$, as defined in \eqref{second-order}, converges to zero, for $100$ equispaced values of $\rho \in [.96,1)$ and initial conditions $(u_0, 0)$ in the interval $u_0 \in [-2,0]$.  The figure suggests that the neighborhood of attraction of the map shrinks to a one-dimensional set as $\rho \rightarrow 1$, as the gray `tornado' regions  indicate initial values $u_0$ for which, after one million iterations, the discrete output $q_n \in \{-1,0,1\}$ has a nontrivial period of fewer than $100$ iterations that persists over an additional one million iterations.  In contrast, the analogous plot for the  asymmetric variant of this map, \eqref{asd}, produces no gray regions.}
\label{fig:quietmap}
\end{center}
\end{figure}

\section{The main result}
As it turns out, the situation changes completely if we apply damping, but not at \emph{every} iteration, only at iterations $n$ for which $u_n \geq 0$ (or the symmetric, when $u_n \leq 0$). As the main contribution of this paper, we introduce the following \emph{asymmetrically-damped} variant of the standard second-order $\Sigma \Delta$ recursion \eqref{second-order},

\begin{eqnarray}
\label{asd}
(u_{n+1}, v_{n+1}) &=& \widetilde{T}\big( (u_n, v_n), f_n \big) := 
\left\{ \begin{array}{ll}
T( \rho(u_n, v_n), f_n ), & u_n \geq 0, \\
T((u_n, v_n), f_n ), & u_n < 0,
\end{array}\right.
\end{eqnarray}
and we shall show that this modification is guaranteed to be quiet.   As far as the author is aware, the asymmetric scheme \eqref{asd} is the first such example of a coarse quantization scheme which obtains second-order approximation accuracy.  Following the argument \eqref{2ndorder:leak}, this asymmetric scheme is second-order if $\rho \geq 1 - \frac{1}{\lambda}$.  Moreover, as a consequence of the asymmetry of the damping, the $u_n$ and in turn the $v_n$, are \emph{forced} to zero at the onset of zero input, no matter how small one sets the difference $1 - \rho$, and independent of the initial conditions $(u_0, v_0)$.  

\paragraph{Numerical Illustration.}  In Figure \ref{fig:2}\subref{fig:subfig1},  we plot the approximation $\widetilde{f}_s(t)$ to an identically zero signal $f(t) \equiv 0$ as produced by the standard second-order $\Sigma \Delta$ recursion with tri-level quantizer, \eqref{second-order}.  The idle tones produced by the periodicity in the output $(q_n)_{n \in \mathbb{N}}$ are easily visible as spikes in the frequency domain of this reconstruction.  In Figure \ref{fig:2}\subref{fig:subfig3}, we plot a reconstruction $\widetilde{f}_q$ produced by the quiet scheme, \eqref{asd}; the spikes in the frequency domain have been clearly smeared out. 


\begin{figure}[!h]
{
\centering
\subfigure[We plot $\widetilde{f}_s(t)$, an approximation to $f(t) \equiv 0$ using the standard second-oder scheme \eqref{second-order}, with oversampling rate $\lambda = 100$ and parameter $\gamma = 2$, and starting from initial conditions $(u_0, v_0) = (.5, .3)$.  This reconstruction is represented in time and also in the frequency domain.  Spikes in the frequency domain, resulting from periodicities in the quantization output, can create idle tones in the reconstructed signal.]{
{\includegraphics[width=6cm]{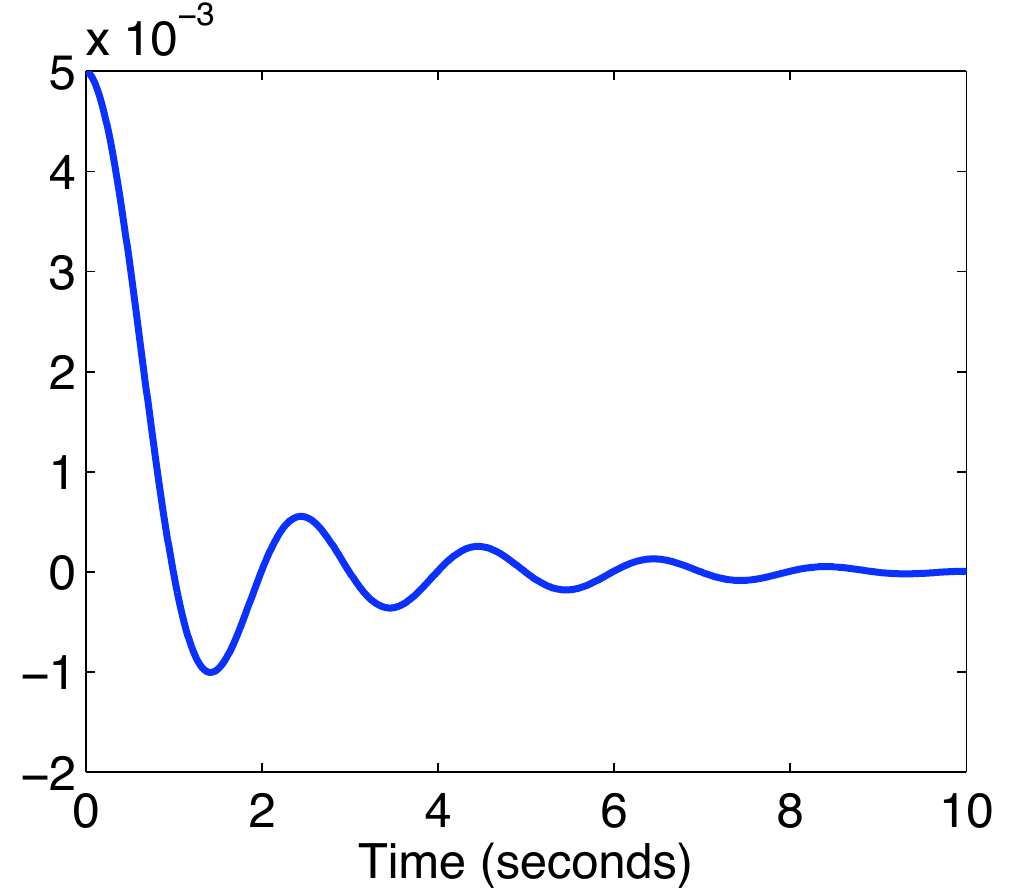}
\label{fig:subfig1}}
{\includegraphics[width=6cm]{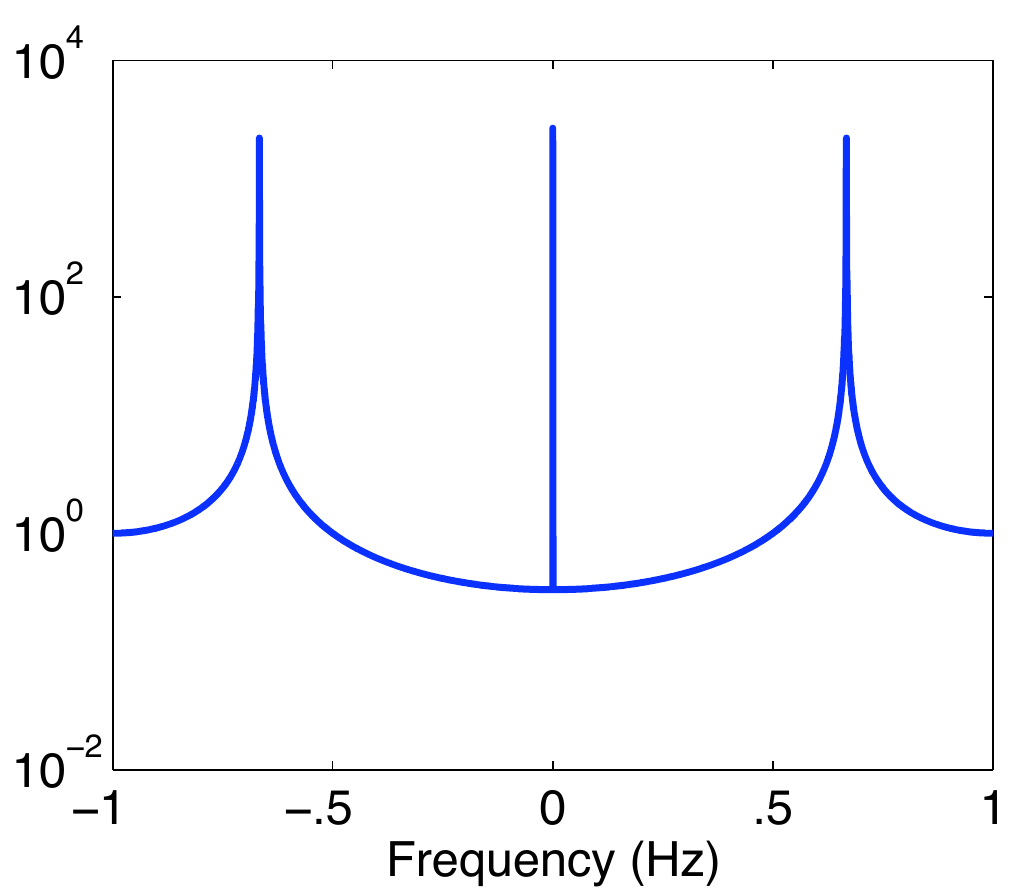}
\label{fig:subfig2}}
}
\subfigure[We plot $\widetilde{f}_q(t)$, an approximation to $f(t) \equiv 0$ using the introduced quiet second-oder scheme \eqref{asd} with oversampling rate $\lambda = 100$,  parameter $\gamma =2$, and damping factor $\rho = .99$, and starting from the same initial condition $(u_0, v_0) = (.5, .3)$, represented in timeand also in frequency.  The spikes in the frequency domain have been smeared out.]{
{\includegraphics[width=6cm]{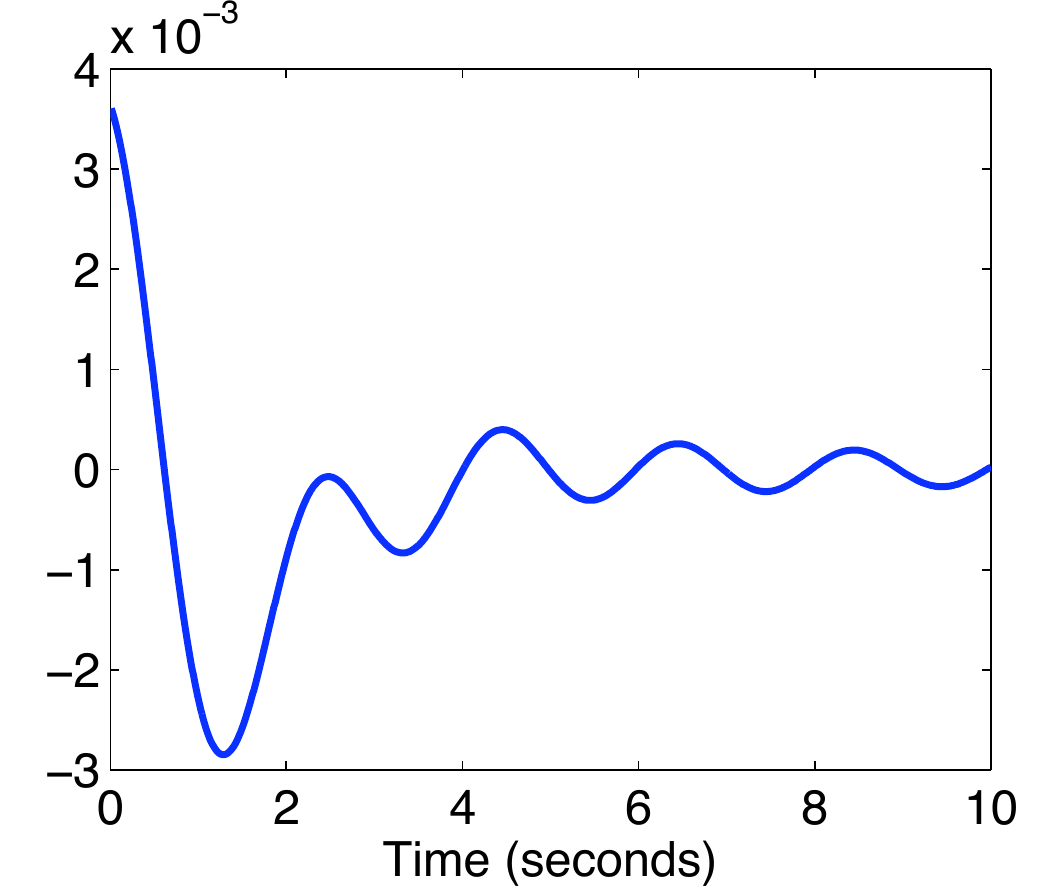}
\label{fig:subfig3}}
{\includegraphics[width=6cm]{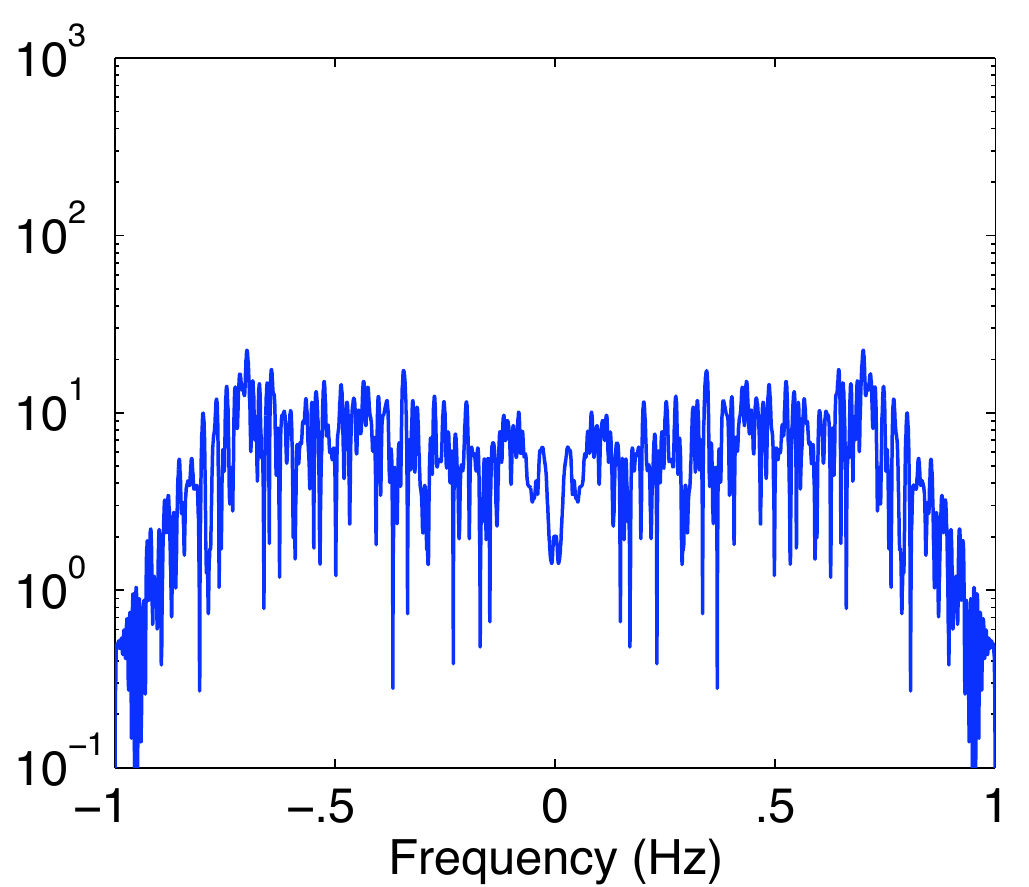}
\label{fig:subfig4}}
}
\label{fig:2}
\caption{Comparison of the standard and asymmetrically-damped second-order $\Sigma \Delta$ quantization schemes.}
}

\end{figure}

\vspace{3mm} 
\noindent The remainder of the paper shall be devoted to proving that the origin is a globally attracting fixed-point of the zero-input map $(u_{n+1}, v_{n+1}) = \widetilde{T}\big( (u_n, v_n), 0 \big)$, as a consequence of the following theorem.
 \begin{theorem}
Consider the linear operator $A: \mathbb{R}^2 \rightarrow \mathbb{R}^2$ given by $A(u, v) = (u, u + v)$, the vectors ${\bf 1} = (1,1), {\bf c} = (\gamma, 1)$, and ${\bf d} = (1,0)$, and the piecewise affine map
\begin{eqnarray}
\label{zeromap:T}
T &:& \mathbb{R}^2 \rightarrow \mathbb{R}^2, \nonumber \\ \nonumber \\
{\bf x}_{n+1} = T {\bf x}_n &:=& \left\{
\begin{array}{ll}
A {\bf x}_n + {\bf 1}, & \textrm{if } \scalprod{{\bf c}}{{\bf x}_n} \leq -1/2 \\

A {\bf x}_n, &  \textrm{if } | \scalprod{{\bf c}}{{\bf x}_n} | < 1/2 \\

A {\bf x}_n - {\bf 1}, & \textrm{if } \scalprod{{\bf c}}{{\bf x}_n} \geq 1/2.
\end{array}\right.
\end{eqnarray}
For any fixed amplification factor  $\gamma \geq 1$ and damping factor $0 \leq \rho < 1$, the origin is a globally attracting fixed point for the asymmetrically-damped piecewise affine map, 
\begin{eqnarray}
\label{asmM}
M &:& \mathbb{R}^2 \rightarrow \mathbb{R}^2 \nonumber \\ \nonumber \\
{\bf x}_{n+1} = M \hspace{.5mm} {\bf x}_n &:=& 
\left\{ \begin{array}{ll}
T(\rho {\bf x}_n), & \scalprod{{\bf d}}{{\bf x}_n} \geq 0, \\
T {\bf x}_n, & \scalprod{{\bf d}}{{\bf x}_n} < 0.
\end{array}\right.
\nonumber
\end{eqnarray} 
\label{mainthm} 
\end{theorem}

Let us pause to discuss an additional application of the main theorem, and of the quiet quantization scheme \eqref{asd}.

\paragraph{Finite Impulse Response filter coefficient quantization.} The reconstruction formulae \eqref{reconstruct} and \eqref{eq:fapprox} for sampling and quantization of bandlimited signals  are particular examples of \emph{discrete-time linear filters}, or implementations of the convolution between an infinite-length input sequence $(x_n)_{n \in \mathbb{N}}$ and a set of coefficients $(c_n)_{n \in \mathbb{N}}$, generating output of the form $y_n = \sum_{j \in \mathbb{N}} c_j x_{n-j}$.  The filter is called a \emph{finite impulse response} (FIR) filter if the set of coefficients is finite-length. The Fourier transform of the output sequence $(y_n)_{n=1}^N$ (in the sense of the discrete-time Fourier transform) is the product of the Fourier transform of the input sequence with the Fourier tranform of the set of coefficients, $\widehat{y}(\omega) = \widehat{c}(\omega) \cdot \widehat{x}(\omega)$.

In some cases, one would like to reduce the hardware complexity of FIR filters by quantizing the \emph{coefficients} $(c_j)_{j=1}^N$ using only a few bits for each coefficient.   If one simply rounds each coefficient value to its nearest quantized level (binary quantization), then the Fourier transform of the set of coefficients may be changed dramatically and therefore the frequency response of the filter is no longer useful. However, in some cases it is only necessary to maintain the filter's frequency response in  lower frequencies, and the rest of the frequencies are not important.

In this case, the original real-valued coefficients of the filter $c_j$ are often modified by sending them through a recursion such as the standard $\Sigma \Delta$ recursion \eqref{eqn:sd},  to obtain a new set of quantized coefficients $q_j$, see \cite{noiseshape}, \cite{noiseshape2}.  The frequency response of the filter is maintained on lower frequencies subject to such quantization; to see this, let us analyze the Fourier transform of the quantization error:
\begin{eqnarray}
\widehat{e}(\omega) &:=& (\widehat{c}(\omega) - \widehat{q}(\omega)) \cdot \widehat{x}(\omega) \nonumber \\
&=&  \widehat{(\Delta^{(m)} v)}(\omega) \cdot \widehat{x}(\omega) \nonumber \\
 &=& (2 \pi i \omega)^m \cdot \widehat{v}(\omega) \cdot \widehat{x}(\omega).
\end{eqnarray}
The problem with this technique is that, while the number of input coefficients  is finite, the quantization output cannot be stopped without causing a large error in the Fourier transform of the quantized coefficients, if the standard noise shaping recursions \eqref{eqn:sd} are applied, \cite{adams}.  However, if one employs instead the quiet noise-shaper, \eqref{asd}, then as a consequence of Theorem \eqref{mainthm}, the tri-level quantization output  is guaranteed to go to zero after the sequence of input coefficients has been exhausted. This means that the set of quantized coefficients is slightly longer than the original coefficient set, as one must wait for the noise-shaper to fall into the `all-zeroes' state, but nevertheless is finite.

In words, the introduced quiet $\Sigma \Delta$ noise-shaper \eqref{asd}, which is guaranteed to be quiet as a result of Theorem \ref{mainthm}, allows a way to quantize coefficients to $3$ levels, such that a finite set of input coefficients can produce a finite set of output coefficients, such that the Fourier transform of the two sets of coefficients match at low frequencies, and diverge at higher frequencies.

\section{Proof of Theorem \ref{mainthm}}
The remainder of the paper is devoted to the proof of Theorem \ref{mainthm}, and is essentially disjoint from the material presented in previous sections. Let us recall that \emph{piecewise affine maps} are discrete dynamical systems of the form
\begin{eqnarray}
\label{PWA}
T &:& \mathbb{R}^m \rightarrow \mathbb{R}^m, \nonumber \nonumber \\
{\bf x}_{n+1} &=& T {\bf x}_n = A_j {\bf x}_n + {\bf b}_j, \hspace{10mm} \textrm{if } {\bf x}_n \in \Omega_j,
\end{eqnarray}
where the sets $\{ \Omega_j \}_{j=1}^{L}$ form a finite partition of the domain $\mathbb{R}^m$, and $A_j \in \mathbb{R}^{m \times m}$. 
It has been recently shown \cite{attractivity} that the \emph{attractivity problem},  or deciding whether or not the origin is a global attracting fixed point of a discrete map, is in general undecidable for piecewise affine maps in dimension $m \geq 2$.   Consequently, there is no universal procedure which can decide, given a generic piecewise affine map in dimension $m \geq 2$,  whether all trajectories converge to zero.  
 In order to address the attractivity problem for a particular piecewise affine map, one must then either verify that certain sufficient conditions hold, or develop a convergence proof for a restricted subclass containing the map of interest.  Sufficient conditions for attractivity generally involve verifying the existence of a \emph{Lyapunov} function $V: \mathbb{R}^m \rightarrow \mathbb{R}^+$ having negative forward difference along trajectories of the map, $\Delta V({\bf x}_{n+1}, {\bf x}_n) = V({\bf x}_{n+1}) - V({\bf x}_n) < 0$.  For example, a recent result in this direction, from \cite{FT}, states that ${\bf x}={\bf 0}$ is a globally attracting fixed point of the piecewise affine map ${\bf x}_{n+1} = T({\bf x}_n)$ if there exists a function $V: \mathbb{R}^m \rightarrow \mathbb{R}^+$ of the form 
\begin{equation}
\label{L}
V({\bf x}) = {\bf x}^T P_j {\bf x}, \hspace{5mm} {\bf x} \in \Omega_j, \hspace{5mm} P_j = P_j^T > 0,
\end{equation}
which has negative forward difference along trajectories $\Delta V({\bf x}_{n+1}, {\bf x}_n) < 0$.
The constraints defining Lyapunov functions of this form may be recast as a set of \emph{linear matrix inequalities},
\begin{eqnarray}
A_j^T P_i A_j - P_j &<& 0,  \nonumber \\
P_i = P_i^T &>& 0, \hspace{5mm} \forall i \in L,
\end{eqnarray}
which must be checked over over all pairs $(i,j)$ for which it is possible that ${\bf x}_k \in \Omega_i$ and ${\bf x}_{k+1} \in \Omega_{i+1}$.  Linear matrix inequalities can be either solved or shown to be nonexistent using standard linear programming solvers and as such, the construction or nonexistence of a Lyapunov function of the form \eqref{L} for a particular piecewise affine map may be determined in polynomial time.  
\\ 
\\
\noindent Unfortunately, Lyapunov functions such as the piecewise positive-definite functions in \eqref{L} which may be tested systematically are very restrictive, and are not applicable to many piecewise affine maps of interest which nevertheless have a global fixed point.  For example, as the reader may check, a Lyapunov function if the form \eqref{L} does not exist for the asymmetric piecewise affine map ${\bf x}_{n+1} = M({\bf x}_n, 0)$ of Theorem \ref{mainthm} which is of interest to us.   Nevertheless, we are still able to prove that the origin is an attractive fixed point of the asymmetric piecewise affine map; the proof is divided into two parts:

\begin{enumerate}
\item In Section $3.1$, we construct a invariant set $S$ for the asymmetrically-damped map $M$, and show that all orbits initialized in this set converge to the origin.  
\item In Section $3.2$, we show, using a Lyapunov function argument, that $S$ is also a global trapping set for $M$.
\end{enumerate} 

\noindent It is not clear that the proof need necessarily be split in two parts as such.  The Lyapunov function constructed in Section $3.1$ only decreases along orbits lying outside the trapping set $S$, and so can not be used to prove convergence directly.  However, there could exist a different Lyapunov function that does decrease along all orbits of the map $M$.  Yet, numerical results such as those in Figure \ref{fig:magnifty_orbit} show a marked change in the behavior of orbits upon entering the invariant region $S$, suggesting an inherent two-phasedness to the dynamics of the system.

\noindent We will use the following notation to distinguish the regions over which $q_n = 1$, $0$, and $-1$, respectively, when $u \geq 0$:

\begin{eqnarray}
\Lambda_1 &=& \{(u,v) \in \mathbb{R}^2: \rho(\gamma u + v)  > 1/2\}, \hspace{10mm} \Lambda_0 = \{(u,v) \in \mathbb{R}^2: \rho |\gamma u + v | \leq  1/2\}, \nonumber \\
\Lambda_{-1} &=& \{(u,v) \in \mathbb{R}^2: \rho (\gamma u + v) < -1/2\}. \nonumber
\end{eqnarray}

\subsection{An invariant set and convergence to the origin} 
In this subsection, we construct a invariant set for the asymmetrically-damped map $M$ of Theorem \ref{mainthm}. 
\begin{proposition}
Consider the regions
\begin{eqnarray}
S^+ &=& \{ (u,v) \in \mathbb{R}^2: -1/2 \leq \gamma u + v \leq 1/2 + \gamma, 0 \leq u < 1 \}, \textrm{and}\nonumber \\
S^- &=& \{ (u,v) \in \mathbb{R}^2: -(1/2 + \gamma) \leq \gamma u + v \leq 1/2, -1 \leq u < 0 \}, \nonumber
\end{eqnarray}
The union of these two regions, $S = S^+ \cup S^-$, as depicted in Figure \ref{fig:stable2}, is a invariant set for the map ${\bf x}_{n+1} = M {\bf x}_n$.  Furthermore, for ${\bf x}_n \in S$:
\begin{itemize}
\item $q_n = 1$ if and only if ${\bf x}_n \in S^+$ and ${\bf x}_{n+1} \in S^-$, 
\item $q_n = -1$ if and only if ${\bf x}_n \in S^-$ and ${\bf x}_{n+1} \in S^+$,  and 
\item If $| q_{k_1} | = | q_{k_2} | = 1$ and $q_k = 0$ for all $k_1 < k < k_2$, then $q_{k_1}$ and $q_{k_2}$ must have alternate signs, i.e. $q_{k_1} q_{k_2} = -1$.
\end{itemize}
\label{invariantregion}
\end{proposition} 
%

\begin{proof}
We begin by showing that $S$ is a invariant set for the asymmetrically damped map ${\bf x}' = M{\bf x}$.  Being $S$ the union of two convex sets, each containing the origin, ${\bf x} \in S$ implies $\lambda {\bf x} \in S$ for $\lambda \in [0,1]$.  As such, it is sufficient to show that  $S$ is invariant for the undamped map ${\bf x}' = T{\bf x}$ of Theorem \ref{mainthm}.  By symmetry of the set $S$ and the map $T$, we can assume without loss that ${\bf x} \in S^+$.  We consider the two cases ${\bf x} \in \Lambda_0$ and ${\bf x} \in \Lambda_1$ separately.  
\begin{enumerate}

\item {\bf $(u,v) \in \Lambda_1$: } In this region, $(u', v') := T(u,v) = ( u - 1,  v + u -1)$, and so in particular $-1 \leq u' < 0$.  Then ${\bf x}' \in S^-$ if 
$$-(1/2 + \gamma) \leq \gamma u' + v' \leq 1/2,$$ 
which is easily verified from the inequalities
$$1/2 \leq \gamma u + v \leq 1/2 + \gamma, \hspace{5mm} \textrm{and} \hspace{3mm} 0 \leq u < 1. $$  

\item {\bf $(u,v) \in \Lambda_0$:  } Now, $(u', v') = (u,  v + u)$, and so $0 \leq u' < 1$.  Also, we have that
$$ -1/2 \leq \gamma u' + v' \leq 1/2 + \gamma, $$
using the inequalities $ -1/2 \leq \gamma u + v \leq 1/2$, $0 < u' < 1$, and $\gamma \geq 1$.
\end{enumerate}
We proceed to the second part of the proposition.  Suppose that $({\bf x}_n)_{n=0}^{\infty}$ is a trajectory contained entirely in $S$, and that $q_1 = q_K = 1$, but $q_n = 0$ for $1 < n < K$.  Because $q_1 = 1$, ${\bf x}_0 \in S^+$ and $u_0 \in (0,1)$.  For the same reason, $u_{K-1} \in (0,1)$.  But $u_1 = \rho u_0 - 1 < 0$, and so $u_{K-1}$, a power of $\rho$ multiplied by $u_1$, must also be negative, leading to a contradiction.  A similar argument rules out the possibility that $q_1 = q_K = -1$ but $q_n = 0$ for $1 < n < K$.
\end{proof}
%


Upon each return to the set $S^+$, ${\bf x}_n$ is `tilted' by the damping ${\bf x} \rightarrow \rho {\bf x}$, creating an imbalance that forces the iterates to zero.
 
 \begin{lemma} \label{utozero}
Suppose that $(u_0, v_0) \in S$, so that $(u_n, v_n) = M (u_0, v_0) \in S$ for all $n \geq 0$ by positive invariance of $S$. The subsequence $(u_{n})_{n \in {\cal I}^+}$ consisting of indices ${\cal I}^+$ for which $u_n \geq 0$ necessarily converges to zero as $n \rightarrow \infty$.    
 \end{lemma}

 \begin{proof}
First observe that the event ${\bf x}_n \in S^+$, or equivalently $u_n \geq 0$, must keep recurring, for if not, then $q_n = 0$ for all $n$ according to Proposition \ref{invariantregion}, and $v_n = v_0 + (n+1)u_0$ diverges.  We may then assume that the index set ${\cal{I}^+}$ represents an infinite subset of the natural numbers, and $(u_n)_{n \in {\cal I}^+}$ an infinite subsequence of $(u_n)_{\mathbb{Z}^+}$.  Moreover, it is clear from the alternating sign pattern of the $q_n$ that $u_{n_{j+1}} = \rho u_{n_j}$ along indices $n_j$ in $\cal I^+$, so that the subsequence $u_{n_j} = \rho^j u_{n_0}$ converges to zero as $n_j \rightarrow \infty$.
 \end{proof}
 
\noindent  Convergence of the subsequence $(u_n)_{n \in {\cal I}^+}$ does not immediately guarantee convergence of the full sequence $(u_n)_{n \in \mathbb{Z}^+}$, as the residual sequence $(u_n)_{n \in \mathbb{Z}^+ \setminus {\cal I}^+}$ could form an infinite subsequence converging to $-1$ as $k \rightarrow \infty$.   However, we can ensure that this pathological situation does not occur.
 
 \begin{proposition}
 If ${\bf x}_0 \in S$, then ${\bf x}_n = M^n {\bf x}_0$ eventually becomes trapped in $S^+ \subset S$.  Moreover, $\| {\bf x}_n \|  \rightarrow 0$ as $n \rightarrow \infty$. 
 \label{tozero!}
 \end{proposition}
 
 \begin{proof}
 As a consequence of Lemma \ref{utozero}, we may fix $\epsilon > 0$, and assume without loss that $(u_0, v_0) \in S^+$ and that $u_0 \leq \epsilon$.  We break the proof into two cases:
 \begin{enumerate}

\item Suppose first that $q_1 = 0$, so that $\gamma u_0 + v_0 \leq 1/2$.  Then $u_1 = \rho u_0$, $v_1 = \rho  (v_0 + u_0)$, and
\begin{eqnarray}
\rho  (\gamma u_1 + v_1) \hspace{1mm} =  \hspace{1mm} \rho ^2 (\gamma u_0 + v_0) + \rho ^2 u_0 \hspace{1mm} \leq \hspace{1mm} \rho/2 + \rho ^2 \epsilon \hspace{1mm}  \leq \hspace{1mm} 1/2, \nonumber
\end{eqnarray}
as long as $\epsilon \leq (1 - \rho)/(2 \rho^2)$.  Consequently, ${\bf x}_n \in S^+$ for all $n$, and
$$
{\bf x}_n = \rho^n A^n {\bf x}_0 = \rho^n (u_0, u_0 + n v_0).
$$
Since $n \rho^n \rightarrow 0$, it follows that $\| {\bf x}_n \| \rightarrow 0.$

\item It remains to consider orbits satisfying ${\bf x}_n \in S^+$ if and only if $q_n = 1$.  Such trajectories are constrained as follows:

\begin{enumerate}
\item If ${\bf x}_n \in S^+$, then $q_{n} = 1$, and \\
\vspace{1mm}
\hspace{10mm} $u_{n+1} = \rho u_n - 1 \leq \rho \epsilon - 1$, \\
\vspace{1mm}
\hspace{10mm} $v_{n+1} = \rho v_n + \rho u_n - 1 \leq \rho v_n + \rho \epsilon - 1$\\

\item  If  ${\bf x}_n \in S^-$ and $q_n = 0$, then ${\bf x}_{n+1} \in S^-$, and\\
\vspace{1mm}
\hspace{10mm} $u_{n+1} =  u_{n} \leq \epsilon - 1$,  \\
\vspace{1mm}
\hspace{10mm} $v_{n+1} = v_{n} + u_{n} \leq v_{n} + \epsilon - 1$ \\ 

\item  If ${\bf x}_n \in S^-$ and $q_n = -1$, then $q_{n+1} \in \{0,1\}$, and \\
\vspace{1mm}
\hspace{10mm} $u_{n+1} = u_{n}  + 1 \leq \epsilon$,  \\
\vspace{1mm}
\hspace{10mm} $v_{n+1} =  v_{n} +  u_{n} + 1 \leq v_{n} + \epsilon$ 
\end{enumerate}

Since $(c)$ cannot occur in successive iterations, we arrive at the period-2 inequality
\begin{equation}
v_{n+2} \leq v_n + 2\epsilon - 1,
\end{equation}
indicating that the iterates $v_n$ diverge. This case, then, cannot occur, and we conclude by Case $1$ that $\| {\bf x}_n \| \rightarrow 0$.
\end{enumerate}  
\end{proof}

\subsection{The invariant set $S$ as a global trapping set}
We show now that the invariant set $S$ constructed in the last section is also a global trapping set for the asymmetric map; in light of Proposition $3$, this guarantees that the origin is a globally attracting fixed point for the map $M$.  
\\
\\
\noindent Before proceeding, we will need the following general lemma:
%

\begin{lemma}
Let $S$ be a invariant set for a discrete map $M$ on a set $X$.  Suppose there exists a nonnegative function $V: X \rightarrow \mathbb{R}^+$ and a parameter $\delta > 0$ with the property that for any ${\bf x} \in X \setminus S$, either $M^k {\bf x} \in S$ or $V(M^k {\bf x}) - V({\bf x}) \leq - \delta$ after a finite time $k$. Then, $S$ is a global trapping set for $M$.
\label{globaltrap}
\end{lemma}

\begin{proof}
Suppose $V$ satisfies the hypotheses, and that ${\bf x} \in X \setminus S$ is such that $M^k {\bf x} \notin S$ for all $k \geq 0$.   Let $c = V({\bf x})$. From the stated hypotheses, $V(M^{k_1} x) \leq c - \delta$ after some finite time $k_1$, and, by induction,  $V(M^{k_{n}}{\bf x}) \leq c - n \delta$ after a finite time $k_n$ for any positive integer $n$.   But then eventually $V(M^{k}{\bf x}) \leq 0$, which is impossible since $V \geq 0$.
\end{proof}

\noindent Lyapunov functions of the form \eqref{L} presented in Proposition \ref{lyap} do not exist for the map $M$, or even the symmetric map $T$, as we invite the reader to verify.  Instead, we follow the approach in \cite{SidongThesis} where trapping sets for the second-order $\Sigma \Delta$ scheme \eqref{second-order} are constructed, in the slightly different setting where quantizer $Q(u) = $sign$(u)$ is considered instead of tri-level quantizer $Q_{tri}$, and  we consider the following Lyapunov function:
\begin{equation}
V(u,v) = u^2 + |2v - u|.
\end{equation}
The motivation for this $V$ is as follows: letting $V^+(u,v) = u^2 + 2v - u$ and $V^-(u,v) = u^2 -2v + u$, it is easily verified that $V(u,v) = \max{ \{V^+(u,v), V^-(u,v) \} }$.  Also, $V^+$ and $V^-$ are the unique functional solutions to the equations
\begin{eqnarray}
\left. V^+(T(u,v)) \right \vert_{(u,v): \gamma u + v \geq 1/2} &=& V^+(u,v),  \nonumber \\ 
\left. V^-(T(u,v)) \right \vert_{(u,v): \gamma u + v \leq -1/2} &=& V^-(u,v). \nonumber
\end{eqnarray}
As it turns out, the set of points ${\bf x}$ for which $V$ may have positive forward difference under iteration of the map $T$, $\Delta V(Tx, x) = V(Tx) - V(x) > 0$, is contained in the invariant set $S$ of  Proposition \ref{invariantregion}. 

 \begin{proposition}
Consider the map $T$ of Theorem \ref{mainthm}, the convex function $V: \mathbb{R}^2 \rightarrow \mathbb{R}^2$ given by $V(u,v) = u^2 + | 2v - u|$, and the set $R = R_1 \cup R_2$, with
 \begin{eqnarray}
 R_1 &=& \{ (u,v) \in \mathbb{R}^2: \hspace{3mm} \gamma u + v \geq 0, \hspace{3mm}  2v + u \leq 1, \hspace{3mm}  u \leq 1/2 \},  \nonumber \\
 R_2 &=& \{ (u,v) \in \mathbb{R}^2 : \hspace{3mm} \gamma u + v < 0, \hspace{3mm} 2v + u \geq -1, \hspace{3mm} u \geq -1/2 \}.
 \nonumber
 \end{eqnarray}
If $\Delta V(T {\bf x}, {\bf x}) > 0$, then ${\bf x} \in R$.  Moreover $R$ is contained in the invariant set $S$ of Proposition \ref{invariantregion}.
 \label{R}
 \end{proposition}
\noindent  We defer the proof of Proposition \ref{R}, which amounts to a straightforward case by case analysis, to Section $5$.  The set $R$ is displayed in Figure \ref{fig:stable3}, along with the invariant set $S$.

 \begin{figure}[!h]
 \begin{center}
 \includegraphics[width=10cm, height=7cm]{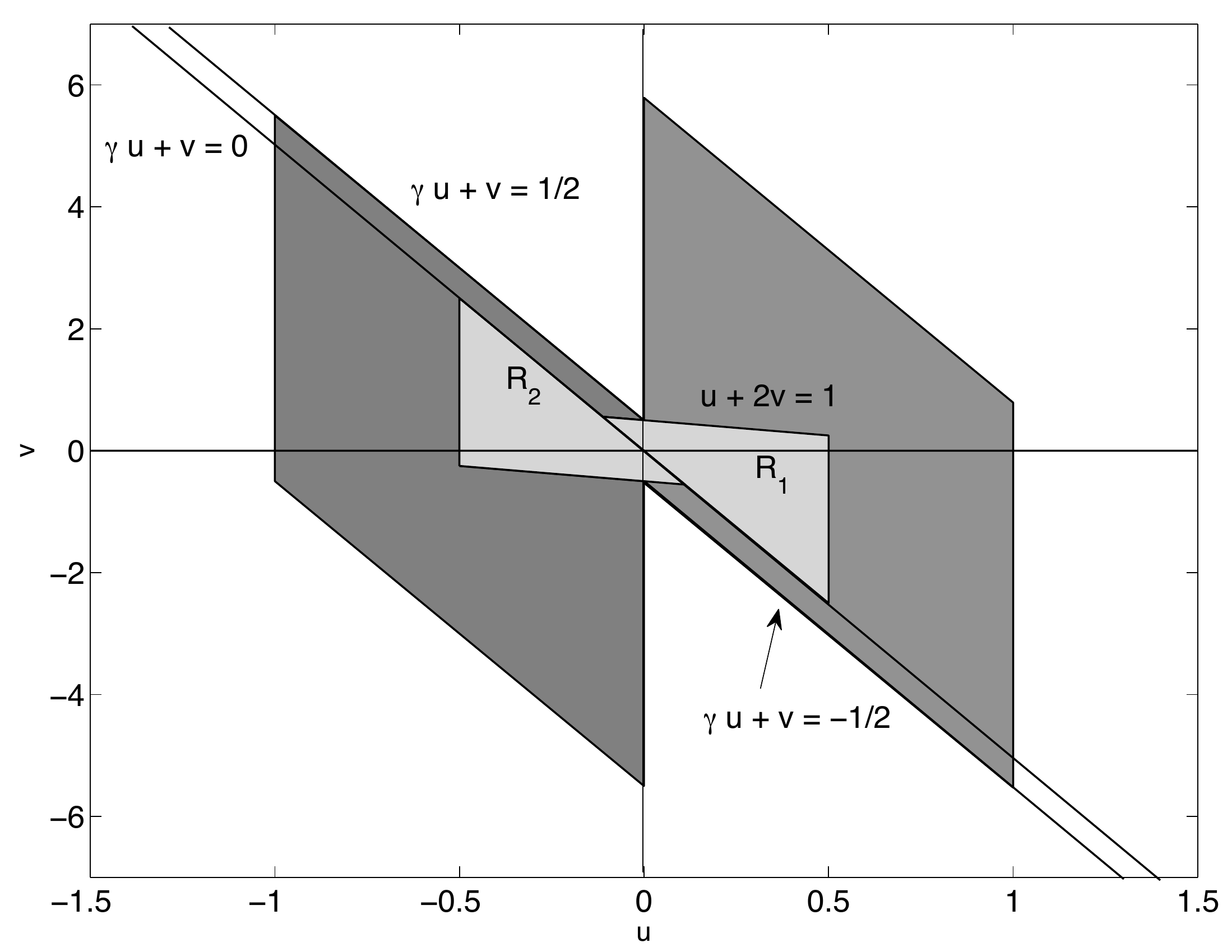}
 \caption{The region $\{(u,v): T(V(u,v)) - T(u,v) > 0\}$ over which the Lyapunov function $V(u,v) = u^2 + | 2v - u|$ may have positive forward difference under the map $T$ is contained in $R$, which is represented by light gray triangles, and is superimposed over the global trapping set $S$.  The parameter used above is $\gamma = 5$.  }
 \label{fig:stable3} 
 \end{center}
 \end{figure}
 
With Proposition \ref{R} and Lemma \ref{globaltrap} in hand,  we are now ready to prove the main result of this section.  

\begin{proposition}
The invariant set $S$ is a global trapping set for the asymmetrically-damped map $M$.
\label{globaltrap2}
\end{proposition}

\begin{proof}
For sufficiently small $\epsilon > 0$, the set $\{ {\bf x}: V({\bf x}) \leq \epsilon \}$ is contained in $S$.  We will verify the conditions in Lemma \ref{globaltrap} for $S$ and $V$, using $\delta := \epsilon (1 - \rho) > 0$.  The proof is split into three cases. 

\begin{enumerate}
\item Suppose first that ${\bf x} = (u,v)$ lies in the positive half plane $u \geq 0$ where $M {\bf x} = T(\rho {\bf x})$.  If ${\bf x}$ is not in the invariant set $S$ but ${\bf x}' = \rho {\bf x}$ is, then $M {\bf x} = {\bf T x}' \in S$ by invariance of $S$ for $T$.  If on the other hand $\rho {\bf x} \notin S$, then
\begin{eqnarray}
V(M {\bf x}) - V({\bf x}) &=&  V\big( T(\rho {\bf x} ) \big) - V({\bf x}) \nonumber \\
&\leq& V(\rho {\bf x}) - V({\bf x}),  \hspace{5mm} \textrm{by Proposition \ref{R}, }   \nonumber \\
&\leq& \rho V({\bf x}) - V({\bf x}), \hspace{5mm}  \textrm{by convexity of $V$ and $V(0) = 0$} \nonumber \\ 
&=& -(1 - \rho) V({\bf x}) \nonumber \\
&\leq& -\epsilon(1 - \rho), \hspace{5mm} \textrm{as $V({\bf x}) > \epsilon$ if ${\bf x} \notin S$} \nonumber \\
&=& -\delta. \nonumber
\end{eqnarray}
Thus, if ${\bf x}$ is in the positive half plane but not in $S$, then either $M {\bf x} \in S$ or $V( M {\bf x}) - V({\bf x}) \leq  - \delta$.  

\item Suppose now that ${\bf x}$ is in the negative half plane, $u < 0$, and also in the set $\Lambda_0 \setminus S$.  We compare $V({\bf x})$ and $V(M {\bf x})$ explicitly:
$$
V(u,v) = u^2 + | 2v - u |, \hspace{5mm} V\big( M(u,v) \big) = V(u, u + v) = u^2 + | 2v + u |.
$$
By inspection of Proposition \ref{invariantregion}, $u$ and $v$ must have opposite signs in this region, $| u | \geq 1$, and $| v | \geq \gamma \geq 1$. It follows that
$$
V \big(M(u,v) \big) - V(u,v) \leq - 2.
$$

\item We have verified the assumptions of Lemma \ref{globaltrap} for all ${\bf x} \in \mathbb{R}^2 \setminus S$ whose trajectories are either eventually contained in $S$ or in the right half plane or in the left half plane intersected with the region $\Lambda_0$.   In fact, \emph{all} ${\bf x} \in \mathbb{R}^2$ can be described as such.  Assume for purposes of contradiction that there exists a point ${\bf x}$ whose \emph{entire trajectory} ${\bf x}_k = M^k {\bf x}$ lies in $\Lambda_{-1} \cap \{(u,v): u < 0\}$, so that $q_k = -1$ for all $k \geq 0$.  But then $u_k = u_0 + k$ becomes arbitrarily large and positive for increasing $k$, an obvious contradiction to the assumption that $u_k < 0$ for all $k$.  This is a contradiction to the assumption that $\gamma u_k + v_k \leq -1/2$, and this case is rendered impossible.  The same argument obviates the possibility that any trajectory ${\bf x}_k = M^k {\bf x}$ lies entirely in $\Lambda_{1} \cap \{(u,v): u < 0\}$.
\end{enumerate}

We conclude that after a finite number of iterations $k$, either $M^k \in S$ or $V(M {\bf x}) - V({\bf x}) \leq -\delta$. 
\end{proof}

\noindent Theorem \ref{mainthm} follows from Proposition \ref{globaltrap2} and Proposition \ref{tozero!}.

\begin{figure}[!h]
\begin{center}
\includegraphics[width=3in]{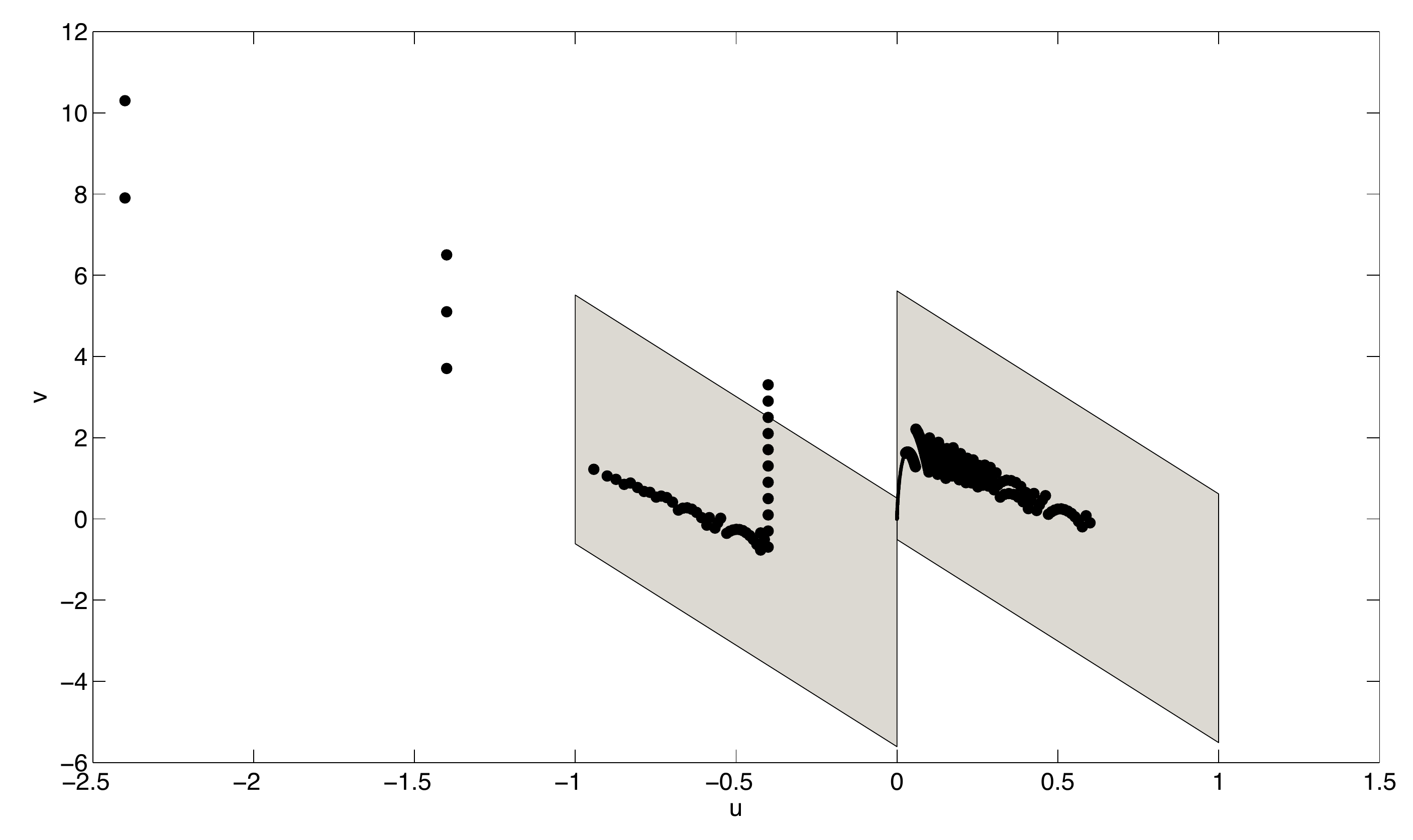}
\includegraphics[width=3in]{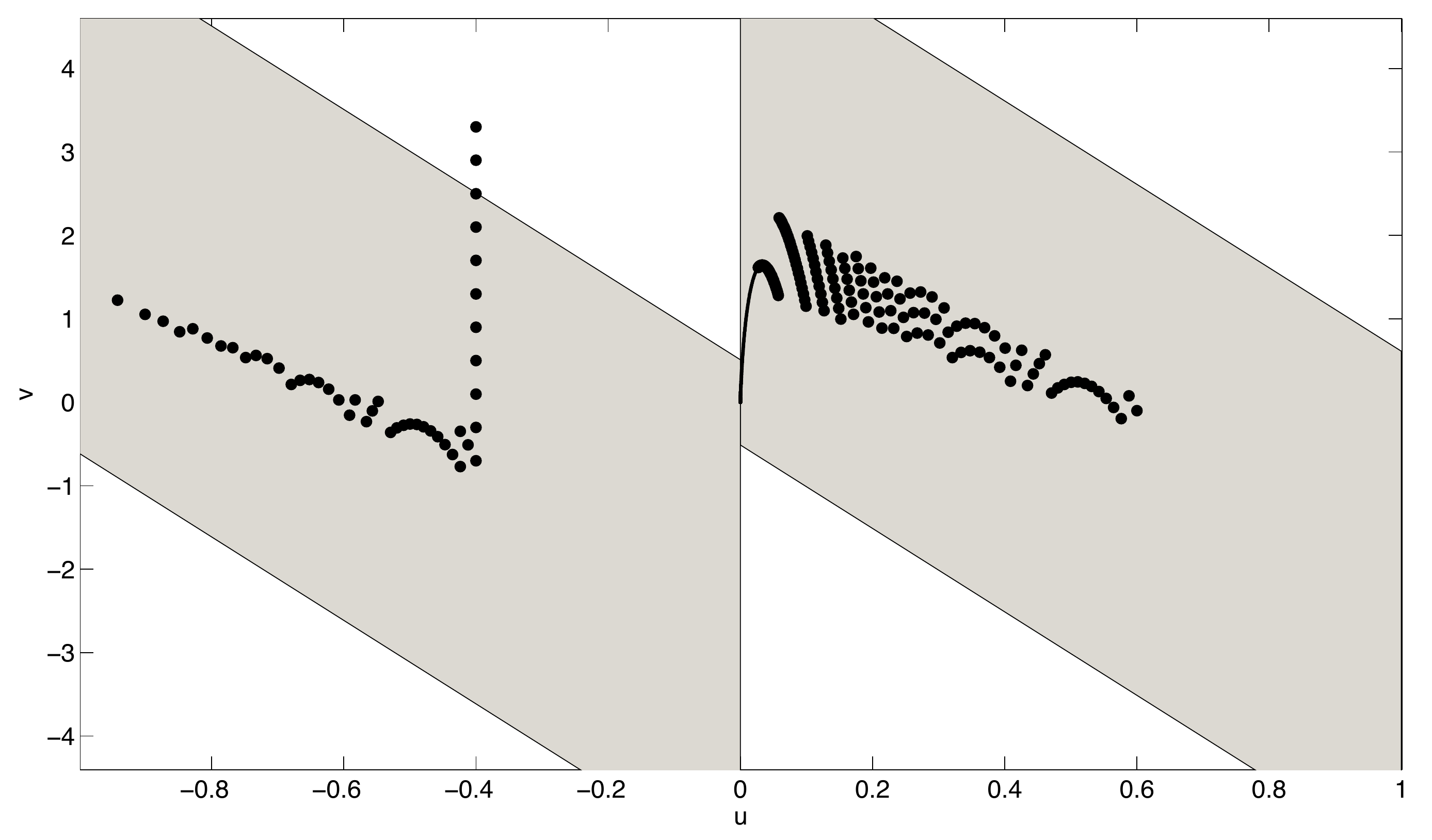}
\caption{ \small Different magnifications of an orbit of the map $M$ for $\rho = .98$ and $\gamma = 5$.  The initial point $(u_0, v_0) = (-3.4, 12.7)$ can be seen in the top image; this point and the first few iterations are outside the trapping set $S$ (gray parallelograms).   Once trapped in $S$, the iterates $(u_n, v_n)$ converge to the globally attracting fixed point $(0,0)$.}
\label{fig:magnifty_orbit}
\end{center}
\end{figure}
 
\section{Conclusions and open problems}
The class of piecewise affine maps in Theorem \ref{mainthm} for which we are able to prove global convergence is very restrictive; indeed, numerical experiments suggest that attractivity holds over a \emph{much} broader class of piecewise affine asymmetric maps.  Interestingly, the most natural generalizations of Theorem \ref{mainthm} from a theoretical perspective are also physically meaningful in the context of $\Sigma \Delta$ quantization.   For instance, the symmetric piecewise affine map $T$ of Theorem \ref{mainthm} corresponding to the original second-order $\Sigma \Delta$ scheme has three affine regions, corresponding to $q=-1$, $0$, and $1$, while the asymmetric map $M$ that is constructed to have a global attracting fixed point has \emph{six} affine regions.  Numerical results suggest that six regions are not necessary, and that attractivity holds for asymmetric maps consisting of three affine regions only including maps of the form,
\begin{eqnarray}
\label{3level}
{\bf x}_{n+1} &:=& \left\{
\begin{array}{ll}
A {\bf x}_n + {\bf 1}, & \textrm{if } \scalprod{{\bf c}}{{\bf x}_n} \leq -\tau \\

A \rho {\bf x}_n, &  \textrm{if } | \scalprod{{\bf c}}{{\bf x}_n} | < \tau \\

A \rho {\bf x}_n - {\bf 1}, & \textrm{if } \scalprod{{\bf c}}{{\bf x}_n} \geq \tau.
\end{array}\right.
\end{eqnarray}
Note that the value $\tau = 1/2$ is no longer constrained; this is because global attractivity of the origin seems to independent of the particular value of $\tau > 0$. This is important in practice if the recursion is to be built into an analog circuit, where thermal fluctuations and other non-idealities cause fluctuations in all of the analog circuit components.  
 
 We conjecture that the origin is also a globally attracting fixed point for four-level  asymmetric maps of the form, 
\begin{eqnarray}
\label{fourlevel2}
(u_{n+1}, v_{n+1})
&=& \left\{
\begin{array}{ll} 
T(\rho_1 u_n, \delta_1 v_n), & \textrm{if } \scalprod{{\bf c}}{(\rho_1 u_n, \delta_1 v_n)} < 0, \\
T(\rho_1\rho_2 u_n, \delta_1\delta_2 v_n), & \textrm{if } \scalprod{{\bf c}}{(\rho_1 u_n, \delta_1 v_n)} \geq 0, 
\end{array}\right.
\end{eqnarray}
where $\rho_1, \rho_2, \delta_1, \delta_2$ are positive parameters less than or equal to $1$, with $\delta_2$ being strictly less than one. This more general framework is a natural extension of Theorem \ref{mainthm}, and also serves as a more realistic model for $\Sigma \Delta$ quantization when implemented in analog circuitry, where, after one clock time, a small amount of \emph{integrator leakage} on each of the two delay elements required to hold each of the states $u_n$ and $v_n$ is unavoidable.  Integrator leakage has the effect of reducing the stored input in the first delay to a fraction $\rho$ of its original value, and reducing the stored input in the second delay by a fraction $\delta$ of its original value; in most circuits of interest, $(\rho, \delta) \in [.95,1]^2$, but the specific leakage factors within this window are generally unknown and may vary slowly in time (see \cite{GregTem}, p. 485, and also \cite{rward}).   The model  \eqref{fourlevel2} also allows for the possibility of inducing leakage on the second delay element only, $(u_n, v_n) \rightarrow (u_n, \rho v_n)$, thus simplifying the recursion.
\\
\\
The class of four-level maps \eqref{fourlevel2} are more complex than the three-level maps \eqref{3level}, but exhibit a much faster speed of convergence of trajectories to the origin.  Indeed, the issue of speed of the convergence is a very important question in itself which we have not yet addressed. For the asymmetric map $M$ in Theorem \ref{mainthm}, the rate of convergence of trajectories to the fixed point ${\bf x} = {\bf 0}$, considered as a function of the discrete time $n$, will invariably decrease as $\rho \rightarrow 1$.  However, it is not clear whether this rate may be constant as a function of the \emph{sampling rate} $\lambda$, when the input $f_n = f(\frac{n}{\lambda})$ represent samples of a band-limited function, and $\rho = 1 - \frac{1}{\lambda}$.  This is an interesting direction for future work. 
\\
\\
\noindent At this point, let us step back to the original motivation for this work, which was to suppress periodicities and quasiperiodicities in the discrete output sequence $q_n$ in the $\Sigma \Delta$ recursion \eqref{second-order} at the onset of stretches of low-amplitude $| f_n^{\lambda} | << 1$.  Our assumption that $f^{\lambda}_n = 0$ over a segment of time is only an idealization of this situation, and the appropriate generalization of `quietness'  for general low-amplitude input is not immediately clear.  However, at the onset of zero-mean input having sufficiently small amplitude and sufficiently high frequency of oscillation, we conjecture that $q_n = 0$ remains a `fixed state'; this  situation occurs for instance at the onset of vanishing input subject to additive noise. 
\\
\\
Finally, it would be interesting to generalize our results to higher dimensional piecewise affine maps, such as the maps corresponding to higher order $\Sigma \Delta$ recursions, \eqref{eqn:sd}. 

 \section{Proof of Proposition \ref{R}}
 
 In this section we prove Proposition \ref{R}, showing that the region where Lyapunov function $V(u,v) = u^2 + | 2v - u |$ has positive forward difference is contained in the set $R = R_1 \cup R_2$ defined by
\begin{eqnarray}
 R_1 &=& \{ (u,v) \in \mathbb{R}^2: \hspace{3mm} \gamma u + v \geq 0, \hspace{3mm}  2v + u \leq 1, \hspace{3mm}  u \leq 1/2 \},  \nonumber \\
 R_2 &=& \{ (u,v) \in \mathbb{R}^2 : \hspace{3mm} \gamma u + v < 0, \hspace{3mm} 2v + u \geq -1, \hspace{3mm} u \geq -1/2 \}.
 \nonumber
 \end{eqnarray}
 
 \begin{proof}
 That $R$ is contained in the trapping set $S$ is straightforward.  Regarding the first part of the proposition, consider the partition of $\mathbb{R}$ into $\Lambda_1 = \{ (u,v) \in \mathbb{R}^2: \gamma u + v \geq 1/2 \}$, $\Lambda^+_0 = \{ (u,v) \in \mathbb{R}^2: 0 \leq  \gamma u + v  < 1/2 \}$, $\Lambda^-_0 = \{ (u,v) \in \mathbb{R}^2: -1/2 \leq  \gamma u + v  < 0 \}$, and $\Lambda_{-1} = \{ (u,v) \in \mathbb{R}^2: \gamma u + v \leq -1/2 \}$.  Suppose that $(u,v) \in \Lambda_1 \setminus R_1$, so that $(u',v') =T(u,v) = (u - 1, u + v - 1)$.  Our first aim is to show that $V(T(u,v)) \leq V(u,v)$ in this situation.  
 \begin{enumerate}
\item {\bf Case 1:} If $2v - u \geq 0$, then $V(u,v) = u^2 + 2v - u$, while $V(u',v') = u^2 - 2u + 1 + | u + 2v -1|$, so 
$$
V(u',v') \leq V(u,v) \Longleftrightarrow | u + 2v - 1| \leq u + 2v  - 1 \Longleftrightarrow u + 2v \geq 1.
$$  
But since $(u,v) \in \Lambda_1 \setminus R_1$, we know that $u > 1/2$, so $u + 2v \geq 2u > 1$, and $V(u',v') \leq V(u,v)$ holds in this case.
\item {\bf Case 2:} If $2v - u \leq 0$, then $V(u,v) = u^2 + u - 2v$, while the expression for $V(u',v')$ remains unchanged; thus,
\begin{equation}
V(u',v') \leq V(u,v) \Longleftrightarrow | u + 2v - 1 | \leq 3u - 2v - 1.
\label{cond}
\end{equation}
We split this case into two subcases:
\begin{enumerate}
\item {\bf Case 2(a)}: If, on the other hand, $u + 2v > 1$, then $|u + 2v - 1| = u + 2v - 1$, and \eqref{cond} simplifies to 
$$
V(u',v') \leq V(u,v) \Longleftrightarrow  u + 2v - 1 \leq 3u - 2v - 1  \Longleftrightarrow 2v \leq u.
$$
 But since $2v - u \leq 0$ by assumption, the result holds in this subcase.

\item {\bf Case 2(b)}: If $u + 2v \leq 1$, then 
\begin{eqnarray}
V(u',v') \leq V(u,v) &\Longleftrightarrow&  -u - 2v + 1 \leq 3u - 2v - 1 \nonumber \\
&\Longleftrightarrow& u \geq 1/2.
\end{eqnarray}
But of course, the condition $u \geq 1/2$ holds throughout $\Lambda_1 \setminus R_1$.
\end{enumerate}
\end{enumerate}

\noindent We have shown thus far that $V(u',v') \leq V(u,v)$ if $(u,v) \in \Lambda_1 \setminus R_1$.  It remains to show that $V(u, u + v) \leq V(u,v)$ if $(u,v) \in \Lambda_0 \setminus R_1$. By inspection of Figure $4$, this region consists of two disjoint sets: $P_1:  \{ (u,v): 0 \leq \gamma u + v \leq 1/2, u + 2v \geq 1, u \leq 0 \}$, and $P_2 :  \{ (u,v): 0 \leq \gamma u + v \leq 1/2, u + 2v \leq -1,  u \geq 0 \}.$

\begin{enumerate}
\item {\bf Case 1: $(u,v) \in P_1$}: As $P_1 \subset \{ (u,v): u \leq 0, v \geq 0 \}$, the restriction $2v - u \geq 0$ trivially holds, and so $h(u,v) = u^2 + 2v - u$, and
\begin{eqnarray}
h(u,u+v) \leq h(u,v) &\Longleftrightarrow& u^2 + | 2(u + v) - u | \leq u^2 + 2v - u  \nonumber \\        &\Longleftrightarrow& |2v + u| \leq 2v - u \nonumber \\
&\Longleftrightarrow& 2v + u \leq 2v - u \nonumber \\
&\Longleftrightarrow& u \leq 0 \nonumber
\end{eqnarray}
which is satisfied by assumption.  
\item {\bf Case 2: $(u,v) \in P_2$ }: $P_2$ is contained in the quadrant $\{ (u,v): u \geq  0, v \leq 0 \}$, and so $2v - u < 0$, $h(u,v) = u^2 - 2v + u$, and
\begin{equation}
h(u,u+v) \leq h(u,v) \Longleftrightarrow | u + 2v | \leq u - 2v \Longleftrightarrow -(u + 2v) \leq u - 2v \Longleftrightarrow u \geq 0, \nonumber
\end{equation}
which again is satisfied by assumption.
\end{enumerate}
By symmetry of the set $R$ and the map $T$, the symmetric result, that $V(T(u,v)) \leq V(u,v)$ if $(u,v) \in \Lambda_{-1}  \setminus R_2$, also holds.

\end{proof}

\subsection*{Acknowledgments}
The author would like to thank Ingrid Daubechies, Sinan Gunturk, and Felix Krahmer for invaluable  discussions on this topic. She is grateful to the American Institute of Mathematics for holding the conference, ``Frames for the finite world: Sampling, coding, and quantization," where this project originated.

\bibliography{QSD_thirddraft}
\bibliographystyle{abbrv}

\end{document}